\documentclass{amsart}
\usepackage{parskip}
\usepackage{enumerate}
\usepackage{amsmath}
\usepackage{amssymb}
\usepackage{amsthm}
\usepackage{graphicx}
\usepackage{mathabx}
\usepackage{float}
\usepackage{tikz}
\usepackage{etex}
\DeclareGraphicsExtensions{.bmp,.png,.pdf,.jpg}
\usepackage{latexsym}
\usepackage{amscd}
\usepackage[all]{xy}
\usepackage{amsmath}
\usepackage[utf8]{inputenc}
\usepackage[pagewise]{lineno}
\DeclareMathOperator{\diam}{diam}

\DeclareMathOperator{\supp}{supp}

\newtheorem{teor}{Theorem}[section]

\newtheorem{lem}[teor]{Lemma}

\newtheorem{prop}[teor]{Proposition}
\newtheorem{defn}[teor]{Definition}

\def\R{\mathbb{R}}

\def\Z{\mathbb{Z}}

\def\N{\mathbb{N}}

\def\LL{\mathcal{L}}

\newcommand{\HH}{\mathcal{H}}

\newcommand{\PP}{\mathcal{P}}

\newcommand{\TT}{\mathcal{T}}
\renewcommand{\S}{\mathcal{S}}
\newcommand{\YY}{\mathcal{Y}}
\newcommand{\GG}{\mathcal{G}}
\newcommand{\minus}{\smallsetminus}
\newcommand{\eps}{\varepsilon}
\newcommand{\vacio}{\varnothing}

\newcommand{\claus}{\overline}

\newcommand{\dmu}{\,d\mu}
\DeclareMathOperator{\INT}{int}
\DeclareMathOperator{\Jac}{Jac}
\providecommand{\abs}[1]{\lvert#1\rvert}
\providecommand{\norm}[1]{\lVert#1\rVert}
\newcommand{\id}{\mathbf{id}}
\newcommand\inner[2]{\langle #1, #2 \rangle}
\usepackage[utf8]{inputenc}
\usepackage[backref=page,colorlinks=true, linkcolor=blue, citecolor=blue, urlcolor=blue]{hyperref}
\usepackage[capitalise]{cleveref} 
\makeatletter

\def\MRbibitem{\@ifnextchar[\my@lbibitem\my@bibitem}

\def\mybiblabel#1#2{\@biblabel{{\hyperref{http://www.ams.org/mathscinet-getitem?mr=#1}{}{}{#2}}}}

\def\myhyperanchor#1{\Hy@raisedlink{\hyper@anchorstart{cite.#1}\hyper@anchorend}}

\def\my@lbibitem[#1]#2#3#4\par{%
  \item[\mybiblabel{#2}{#1}\myhyperanchor{#3}\hfill]#4%
  \@ifundefined{ifbackrefparscan}{}{\BR@backref{#3}}%
  \if@filesw{\let\protect\noexpand\immediate
    \write\@auxout{\string\bibcite{#3}{#1}}}\fi\ignorespaces%
}

\def\my@bibitem#1#2#3\par{%
  \refstepcounter\@listctr
  \item[\mybiblabel{#1}{\the\value\@listctr}\myhyperanchor{#2}\hfill]#3%
  \@ifundefined{ifbackrefparscan}{}{\BR@backref{#2}}%
  \if@filesw\immediate\write\@auxout
    {\string\bibcite{#2}{\the\value\@listctr}}\fi\ignorespaces%
}

\makeatother

\title[Decay of Correlations for some NUH Attractors]{Decay of Correlations for some Non-Uniformly Hyperbolic Attractors}
\author{Sebastian Burgos }
\date{}

\begin{document}

\begin{abstract}
We study the decay of correlations for certain dynamical systems with non-uniformly hyperbolic attractors, which natural invariant measure is the Sinai-Ruelle-Bowen (SRB) measure. The system $g$ that we consider is produced by applying the slow-down procedure to a uniformly hyperbolic diffeomorphism $f$ with an attractor. Under certain assumptions on the map $f$ and the slow-down neighborhood, we show that the map $g$ admits polynomial upper and lower bounds on correlations with respect to its SRB measure and the class of H\"older continuous observables. Our results apply to the Smale-Williams solenoid, as well as its sufficiently small perturbations.
\end{abstract}

\maketitle

\section{Introduction}

In smooth dynamical systems, a classical problem known as \textit{smooth realization problem} asks if given a compact smooth manifold $M$, one can construct a volume preserving diffeomorphism $f$ with prescribed ergodic properties. Katok \cite{K} constructed an example (called the Katok map) of an area preserving $C^{\infty}$ Bernoulli diffeomorphism with nonzero Lyapunov exponents first on the two dimensional torus, then transferred the construction to the sphere, to the unit disk, and finally to any surface. Brin, Feldman and Katok \cite{BFK} showed that every manifold of dimension greater than one carries a Bernoulli diffeomorphism, then Brin \cite{B} extended this result by constructing a volume preserving $C^{\infty}$ Bernoulli diffeomorphism on any manifold of dimension at least 5, with all but one nonzero Lyapunov exponents. Finally, Dolgopyat and Pesin \cite{DP} extended this result by constructing a Bernoulli diffeomorphism with nonzero Lyapunov exponents on any smooth compact Riemannian manifold of dimension at least 2. Another ergodic property of interest is \textit{mixing}, which amounts to have correlations decaying to zero. Thus, one can ask about the speed at which these correlations decay. Regarding this question Pesin, Senti and Shahidi \cite{PSS} showed that any smooth compact connected and oriented surface admits an area preserving $C^{1+\beta}$ diffeomorphism with nonzero Lyapunov exponents which is Bernoulli and has polynomial decay of correlations.

The constructions mentioned above involve the so-called \textit{slow-down procedure}. Roug-hly speaking, it goes as follows: consider a uniformly hyperbolic map $f$ such that around a fixed point $p$, it is the time-one map of a flow given by a linear vector field. In a very small neighborhood of $p$, we perturb the system of equations in order to slow down trajectories and break the uniform hyperbolicity of the map, producing a new map $g$ which coincides with the map $f$ outside of a neighborhood of $p$, and is non-uniformly hyperbolic. We study the decay of correlations of dissipative systems, i.e. maps with non-uniformly hyperbolic attractors, which are produced via the slow-down procedure. This is the main difference between our setting and all the examples mentioned above (conservative systems). In our setting the natural invariant measure is the Sinai-Ruelle-Bowen (SRB) measure (See Section \ref{defns}). We provide an example in dimension 3: the Smale-Williams solenoid.

To the best of our knowledge, there are not many results involving statistical properties of SRB measures in genuine attractors. It is known that any $C^2$ transitive Anosov diffeomorphism $f$ admits a unique invariant measure $\mu$ which has absolutely continuous conditional measures on unstable manifolds, and the system $(f,\mu)$ has exponential decay of correlations with respect to H\"older observables (\cite{Sinai}, \cite{Ruelle}, \cite{Bowen}). Benedicks and Young discussed in \cite{BY} the existence of SRB measures with exponential decay of correlations for H\'enon maps. Later Castro discussed in \cite{C} the same question for some partially hyperbolic diffeomorphisms. Hatomoto studied in \cite{H} a $C^{1+\alpha}$ partially hyperbolic diffeomorphism which restriction on one dimensional center unstable direction behaves as Manneville-Pomeau map, and showed that admits a unique ergodic SRB measure with polynomial upper bound on correlations with respect to H\"older observables.

We now present the setting and the main results. Starting with a $C^{1+\alpha}$ diffeomorphism $f\colon M\to M$ of a smooth manifold with a topological attractor $\Lambda$, we apply the slow-down procedure to obtain a new map $g$ with its own attractor $\Lambda_g$ (see Section \ref{Setting} for more details). Climenhaga, Dolgopyat and Pesin \cite{CDP} studied this setting and showed that the map $g$ has an SRB measure $\mu$ supported on $\Lambda_g$. Zelerowicz \cite{Z} studied thermodynamic formalism for the slow-down map $g$ with the family of potentials $\varphi_t(x):=-t\log\abs{dg|_{E^u(x)}}$, and showed that the SRB measure of $g$ is unique, and it is precisely the equilibrium measure of the potential $\varphi_1$. Moreover, under an additional condition on $f$, there is $t_0<0$ such that for $t\in(t_0,1)$ the potential $\varphi_t$ has a unique equilibrium measure $\mu_t$, which has exponential decay of correlations. When $t=1$ the potential $\varphi_1$ has two ergodic equilibrium measures, the unique SRB measure and the Dirac measure at the fixed point $p$. A natural question to ask is when we consider the SRB measure, whether we still have exponential decay of correlations or it drops to some slower decay. We use the setting and some ideas and techniques from \cite{CDP}, \cite{Z} and \cite{PSS} to prove the main results of the paper.

\begin{teor} Let $f$ be a $C^{1+\alpha}$ diffeomorphism of a compact manifold $M$ with a uniformly hyperbolic attractor $\Lambda$, satisfying (A1)--(A5) (see Section \ref{Setting}). Then for sufficiently small slow-down neighborhood, the slow-down map $g\colon\Lambda_g\to\Lambda_g$ admits polynomial upper bound on correlations with respect to its SRB measure and the class of H\"older continuous observables on $\Lambda_g$.
\end{teor}

\begin{teor}
Under the assumptions of the previous theorem, assume in addition that $\alpha\in(0,1/2)$, $\gamma>0$ is fixed and  $\beta>\gamma$ in (A4) (see Section \ref{Setting}) is sufficiently large. Then for sufficiently small slow-down neighborhood, the slow-down map $g\colon\Lambda_g\to\Lambda_g$ admits polynomial lower bound on correlations with respect to its SRB measure and a subclass of the H\"older continuous observables on $\Lambda_g$.
\end{teor}

A very useful technique to establish decay of correlations is called an \textit{inducing scheme} (see \cite{PSZ1} and \cite{SZ}), which has applications to so-called \textit{Young diffeomorphisms} (see Section \ref{prelim} for more details). Pesin, Senti and Zhang \cite{PSZ} studied thermodynamic formalism for the Katok map, and showed that it is a Young diffeomorphism. Later, Pesin, Senti and Shahidi \cite{PSS} used this structure to establish polynomial decay of correlations for a class of maps depending on the surface. For the attractor setting, Zelerowicz \cite{Z} showed that the map $g$ is a Young diffeomorphism, and so we have the main tool to establish decay of correlations.

The paper is organized in the following way: in Section \ref{defns} we fix notation and definitions that we will use throughout the paper, in Section \ref{Setting} we explain the setting in which we will work and we state the main results in more detail, in Section \ref{solenoid} we provide an example in dimension 3 that fits in our setting, and finally Section \ref{prelim} and Section \ref{Proofs} consist of the proofs of the main results.

\textit{Acknowledgements.} I would like to thank my advisor, Yakov Pesin, for proposing the problem and for very useful comments and suggestions. I would like to thank Sebastian Pavez for many discussions about the solenoid example, and Eduardo Reyes for his valuable comments.

\section{Definitions and Notations}\label{defns}

\subsection{Decay of Correlations} Let $(X,T,\mu)$ be a measure preserving dynamical system. Let $\HH$ be a class of real valued functions on $X$. For $h_1,h_2\in\HH$ and $n\geq 1$, the sequence of correlations between $h_1$ and $h_2$ with respect to the system $(T,\mu)$ is defined by
$$C^{\mu}_n(h_1,h_2):=\int (h_1\circ T^n)h_2d\mu-\int h_1d\mu\int h_2d\mu.$$

We say that the system $(X,T,\mu)$ has \textit{exponential decay of correlations with respect to the class $\HH$} if there exists $\gamma>0$ such that for any $h_1,h_2\in\HH$ and $n\geq 1$,
$$\abs{C^{\mu}_n(h_1,h_2)}\leq Ce^{-n\gamma},$$
where $C=C(h_1,h_2)>0$ is a constant.\\

We say that the system $(X,T,\mu)$ has \textit{polynomial decay of correlations} (more precisely, \textit{polynomial upper bound on correlations}) with respect to the class $\HH$ if there exists $\gamma '>0$ such that for any $h_1,h_2\in\HH$ and any $n\geq 1$,
$$\abs{C^{\mu}_n(h_1,h_2)}\leq C'n^{-\gamma'},$$
where $C'=C'(h_1,h_2)>0$ is a constant.\\

Similarly, we say that the system $(X,T,\mu)$ has \textit{lower bound on correlations} with respect to the class $\HH$ if there exists $\gamma ''>0$ such that for any $h_1,h_2\in\HH$ and any $n\geq 1$,
$$\abs{C^{\mu}_n(h_1,h_2)}\geq C''n^{-\gamma''},$$
where $C''=C''(h_1,h_2)>0$ is a constant.

\subsection{Young diffeomorphisms.}\label{young} We define Young diffeomorphisms following \cite{Z}. Let $f\colon M\to M$ be a $C^{1+\alpha}$ diffeomorphism of a compact smooth Riemannian manifold $M$, and assume that $\Lambda\subset M$ is a topological attractor for $f$. 

An embedded disk $V\subset M$ is called an \textit{unstable disk} (respectively \textit{stable disk}) if for all $x,y\in V$ we have $d(f^{-n}(x),f^{-n}(y))\to 0$ (respectively $d(f^n(x),f^n(y))\to 0$) as $n\to\infty$. A collection of embedded disks $\Gamma^u=\{V^u\}$ is called a \textit{continuous family of $C^1$ unstable disks} if there is a homeomorphism $\Theta\colon K^s\times D^u\to\bigcup V^u$ satisfying:
\begin{itemize}
    \item $K^s\subset M$ is a Borel subset and $D^u\subset\R^d$ is the closed unit disk for some $d<\dim M$;
    \item $x\mapsto\Theta|_{\{x\}\times D^u}$ is a continuous map from $K^s$ to the space of $C^1$ embeddings of $D^u$ into $M$ which can be extended continuously to a map of the closure $\claus{K^s}$;
    \item $V^u=\Theta(\{x\}\times D^u)$ is an unstable disk.
\end{itemize}

A \textit{continuous family of stable disks} can be defined analogously.\\

We say that a set $P\subset\Lambda$ has \textit{hyperbolic product structure} if there is a continuous family $\Gamma^u=\{V^u\}$ of unstable disks $V^u$ and a continuous family $\Gamma^s=\{V^s\}$ of stable disks such that
\begin{itemize}
    \item $\dim V^s+\dim V^u=\dim M$;
    \item the $V^u$-disks are transversal to the $V^s$-disks by an angle uniformly bounded away from zero;
    \item each $V^u$-disk intersects each $V^s$-disk at exactly one point;
    \item $P=(\bigcup V^u)\cap(\bigcup V^s)$.
\end{itemize}

A subset $P_0\subset P$ is called an \textit{s-subset} if it has hyperbolic product structure and is defined by the same family $\Gamma^u$ of unstable disks as $P$ and a continuous subfamily $\Gamma_0^s\subset\Gamma^s$ of stable disks. A \textit{u-subset} is defined analogously.\\

The map $f$ is called a \textit{Young diffeomorphisms} if it satisfies the following conditions:

\begin{enumerate}
    \item[(C0)] There exists $P\subset\Lambda$ with hyperbolic product structure.
    \item[(C1)] There exists a countable collection of continuous subfamilies $\{\Gamma^s_i\}_{i\in\N}$ of stable disks and positive integers $\tau_i$, $i\in\N$ such that the $s$-subsets
    $$P^s_i:=\bigcup_{V\in\Gamma_i^s}(V\cap P)$$
    are pairwise disjoint and satisfy
    \begin{enumerate}
        \item Invariance: for every $x\in P_i^s$
        $$f^{\tau_i}(V^s(x))\subset V^s(f^{\tau_i(x)}),\qquad f^{\tau_i}(V^u(x))\supset V^u(f^{\tau_i(x)}),$$
        where $V^{u,s}(x)$ denotes the (un)stable disk containing $x$;
        \item Markov Property: the set $P^u_s=f^{\tau_i}(P^s_i)$ is a $u$-subset of $P$ such that for all $x\in P^s_i$
        $$f^{-\tau_i}(V^s(f^{\tau_i}(x))\cap P^u_i)=V^s(x)\cap P,$$
        $$f^{\tau_i}(V^u(x)\cap P^s_i)=V^u(f^{\tau_i}(x))\cap P.$$
        \end{enumerate}
        For any $x\in P^s_i$ define the \textit{inducing time} by $\tau(x):=\tau_i$ and the \textit{induced map} $F\colon\bigcup_{i\in\N}P^s_i\to P$ by
        $$F|_{P_i^s}:=f^{\tau_i}|_{P_i^s}.$$
    
    \item[(C2)] There is $a\in(0,1)$ such that for any $i\in\N$ we have:
    \begin{enumerate}
        \item For any $x\in P_i^s$ and $y\in V^s(x)$,
        $$d(F(x),F(y))\leq ad(x,y);$$
        \item For any $x\in P_i^s$ and $y\in V^u(x)\cap P_i^s$,
        $$d(x,y)\leq ad(F(x),F(y)).$$
    \end{enumerate}
    
    For $x\in P$ denote
    $$E^u(f^k(x)):=T_{f^k(x)}f^k(V^u(x))=df_x^k(T_xV^u(x))$$
    and denote the restriction of the Jacobian of $f$ to $E^u$ by $J^uf(x):=\det df|_{E^u(x)}$. The definition of $J^uF(x)$ is analogous.
    
    \item[(C3)] There exist $c>0$ and $0<\beta<1$ such that:
    \begin{enumerate}
        \item[(a)] For all $n\geq 0$, $x\in F^{-n}(\bigcup_{i\in\N}P^s_i)$ and $y\in V^s(x)$ we have
        $$\log\frac{J^uF(F^nx)}{J^uF(F^ny)}\leq c\beta^n;$$
        \item[(b)] For any $i_0,\ldots,i_n\in\N$, $F^k(x),F^k(y)\in P^s_{i_k}$ for $0\leq k\leq n$ and $y\in V^u(x)$ we have
        $$\log\frac{J^uF(F^{n-k}x)}{J^uF(F^{n-k}y)}\leq c\beta^k.$$
    \end{enumerate}
    
    Let $m_{V^u}$ be the leaf volume on $V^u$.
    \item[(C4)] For every $V^u\in\Gamma^u$ one has
    $$m_{V^u}((P\minus\cup P^s_i)\cap V^u)=0\quad\text{ and }\quad m_{V^u}(V^u\cap P)>0.$$
    \item[(C5)] There exists $V^u\in\Gamma^u$ such that
    $$\sum_{i=1}^{\infty}\tau_i m_{V^u}(P^s_i)<\infty.$$
    
    Define
    $$scl(P^s_i):=\bigcup_{x\in\claus{P^s_i\cap V^u}}V^s(x)\cap P.$$
    
    One can easily see that the definition of $scl(P^s_i)$ does not depend on the choice of $V^u$.
    \item[(C6)] The set $\bigcup_{i\in\N}(scl(P^s_i)\minus P^s_i)$ supports no invariant measure that gives positive weight to any set which is open in the induced topology of $\Lambda$.
\end{enumerate}

\subsection{Hyperbolic and SRB measures}

Let $g$ be a $C^{1+\alpha}$ diffeomorphism of a compact smooth manifold $M$. A $g$-invariant measure $\mu$ is called \textit{hyperbolic} if $g$ has nonzero Lyapunov exponents $\mu$-almost everywhere.

For a regular set $Y_{\ell}$ (see \cite{BP1} or \cite{BP2} for the definition) of positive $\mu$-measure and sufficiently small $r>0$, define $Q_{\ell}(x)=\bigcup_{w\in Y_{\ell}\cap B(x,r)}V^u(w)$ for every $x\in Y_{\ell}$, where $V^u(w)$ is the local unstable manifold through $w$. Let $\xi(x)$ be the partition of $Q_{\ell}(x)$ by these manifolds, and let $\mu^u(w)$ be the conditional measure on $V^u(w)$ generated by $\mu$ with respect to $\xi(x)$.

\begin{defn}
A hyperbolic invariant measure $\mu$ is called an SRB measure if for any regular set $Y_{\ell}$ of positive measure and almost every $x\in Y_{\ell}$, $w\in Y_{\ell}\cap B(x,r)$, the conditional measure $\mu^u(w)$ is absolutely continuous with respect to the leaf volume $m_{V^u(w)}$ on $V^u(w)$.
\end{defn}
\section{Setting and Main Results}\label{Setting} 

We follow \cite[Section 1]{Z}. Let $M$ be a $d$-dimensional compact smooth Riemannian manifold and $U\subset M$ an open set. Let $f\colon U\to M$ be a $C^{1+\alpha}$ diffeomorphism onto its image with $\claus{f(U)}\subset U$, where $\alpha\in(0,1)$. Let $\Lambda:=\bigcap_{n=0}^{\infty}\claus{f^n(U)}$ be an attractor for $f$ with $NW(f)=\Lambda$, and assume the following:

\begin{enumerate}
    \item[(A1)] The map $f|_{\Lambda}$ is topologically transitive.

    \item[(A2)] The set $\Lambda$ is hyperbolic for $f$, i.e. for every $x\in\Lambda$, there is a splitting of the tangent space $T_xM=E^u_f(x)\oplus E^s_f(x)$ with $d_xf(E_f^u(x))=E^u_f(f(x))$ and $d_xf(E^s_f(x))=E^s_f(f(x))$ satisfying
    $$\norm{d_xf(v^u)}\geq\nu\norm{v^u}\quad\text{ for all }\quad v^u\in E^u_f(x),$$
    $$\norm{d_xf(v^s)}\leq\zeta\norm{v^s}\quad\text{ for all }\quad v^s\in E^s_f(x),$$
    for some $\nu>1$ and $\zeta <1$.
    
    \item[(A3)] For every $x\in\Lambda$, the unstable subspace $E^u(x)$ is one-dimensional;
    
    \item[(A4)] The map $f$ has a fixed point $p\in\Lambda$, and there is a neighborhood $V$ of $p$ in which $f$ is the time-one map of the flow generated by a linear vector field $\dot{x}=Ax$, where $A=A_u\oplus A_s$, $A_u=\gamma\id_u$ and $A_s=-\beta\id_s$, for some $\beta>\gamma>0$.
    
    \textit{Remark:} In a neighborhood of $V$ we identify $E^u(p)\oplus E^s(p)$ with $\R\oplus\R^{d-1}$ and $p$ with $0$.
    
    \item[(A5)] The map $f$ lies in a sufficiently small $C^1$-neighborhood of a map $f_0$ for which the SRB measure and the measure of maximal entropy coincide, and which has uniform expansion, i.e. there is $\lambda>1$ such that 
    $$\norm{d_xf_0v}=\lambda\norm{v}\qquad\text{ for }v\in E^u_{f_0}(x).$$ 
    
    \textit{Remark:} Since we are choosing $f$ sufficiently close to $f_0$, our assumptions for $f$ force us to assume that the unstable subspaces of $f_0$ are also one-dimensional.
\end{enumerate}

\textit{Remark:} Since $f|_{\Lambda}$ is topologically transitive and there is a fixed point, the spectral decomposition theorem implies that $f|_{\Lambda}$ is in fact topologically mixing.

We now apply the slow-down procedure to the map $f$. Fix $0<r_0<r_1<1$ so that $B(0,r_1)\subset V$. Let $\psi\colon[0,1]\to[0,1]$ be a $C^{1+\alpha}$ function satisfying
\begin{enumerate}
    \item[(i)] $\psi(r)=r^{\alpha}$ for every $r\leq r_0$;
    \item[(ii)] $\psi(r)=1$ for every $r\geq r_1$;
    \item[(iii)] $\psi'(r)> 0$.
\end{enumerate}

Let $\varphi_t$ be the flow generated by the system $\dot{x}=\psi(\norm{x})Ax$ on $V$. Define $g\colon U\to M$ by
$$g(x):=\left\{\begin{array}{lcc}
    \varphi_1(x) & , & x\in V\\
     f(x) & , & x\in U\minus V 
\end{array}.\right.$$

The map $g$ is of class $C^{1+\alpha}$, we have $g(U)=f(U)$ and in particular $\claus{g(U)}\subset U$. Then, the set $\Lambda_g:=\bigcap_{n=0}^{\infty}\claus{g^n(U)}$ is a topological attractor for $g$.

The map $g$ is called a \textit{slow-down map} of the map $f$. The same way Pesin, Senti and Shahidi \cite{PSS} studied decay of correlations for the Katok map, we study decay of correlations for the map $g$, also constructed via the slow down procedure. However, in our case we cannot work with volume as the natural measure, as the attractor may have zero volume. It was proven in \cite[Theorem 2.4]{CDP} that the map $g$ has an SRB measure $\mu$ supported on $\Lambda_g$. This is the natural measure with respect to which we study decay of correlations of the map $g$.

We now state the main results. Let $\HH$ be the class of real valued H\"older observables on $\Lambda_g$. For a sequence $\mathcal{Y}$ of nested subsets $Y_1\subset Y_2\subset \cdots\Lambda_g$, denote by $\mathcal{G}(\mathcal{Y})$ the set of $h\in\HH$ for which there is $k=k(h)$ with $\supp(h)\subset Y_k$. 

\begin{teor}\label{mainresult1} For $r_1>0$ sufficiently small, the map $g\colon\Lambda_g\to\Lambda_g$ admits a polynomial upper bound on correlations with respect to $\mu$ and the class $\HH$. More precisely: there is $s_1>0$ such that for any $h_1,h_2\in\HH$ we have 
    $$\abs{C^{\mu}_n(h_1,h_2)}\leq C_1n^{-s_1},$$
    where $C_1=C_1(h_1,h_2)>0$.

\end{teor}
\begin{teor}\label{mainresult2}
      Suppose $\alpha\in(0,1/2)$ and $\beta>\beta_0$ for some sufficiently large $\beta_0>\gamma$ (recall $\beta$ and $\gamma$ defined in (A4)). Then for $r_1>0$ sufficiently small, the map $g\colon\Lambda_g\to\Lambda_g$ admits a polynomial lower bound with respect to $\mu$ and a subclass of $\HH$: there is a nested sequence $\YY=\{Y_i\}$ of subsets of $\Lambda_g$  and $s_2>0$ such that for any $h_1,h_2\in\GG(\YY)$ with $\int h_1\dmu\int h_2\dmu>0$,
        $$\abs{C_{n}^{\mu}(h_1,h_2)}\geq C_2n^{-s_2},$$
        where $C_2=C_2(h_1,h_2)>0$.

\end{teor}

\section{Example: The Smale-Williams solenoid}\label{solenoid}
\vspace{0.4cm}
An example of the slow-down procedure in dimension 2 was provided by Katok \cite{K} when constructing the Katok map by perturbing a hyperbolic toral automorphism. We now provide an example in dimension 3, showing that the Smale-Williams solenoid fits in the setting from Section \ref{Setting}, and hence our results apply to these examples.

Let $\TT:=S^1\times D$ be the solid torus, where $S^1=\R/\Z$ and $D$ is the closed unit disk in $\R^2$.\\

\subsection{The general three-dimensional solenoids.} Let $F\colon\TT\to\TT$ be a $C^1$ embedding of the form
$$F(t,x,y)=(h(t),\lambda_1(t) x+z_1(t),\lambda_2(t) y+z_2(t)),$$
where
$$h\colon S^1\to S^1,\quad\lambda_1,\lambda_2\colon S^1\to(0,1) ,\quad z_1,z_2\colon S^1\to(-1,1)$$
are $C^1$ maps and $h$ is expanding (i.e. $dh/dt>1$). Then $m:=\deg h\geq 2$ and $F$ stretches the solid torus in the $S^1$-direction. We may also impose some conditions on $\abs{\lambda_i}$ and $\abs{z_i}$ so that $F(\TT)\subset\INT\TT$. Then, the image $F(\TT)$ is wrapped around inside $\TT$ exactly $m$ times. We define the solenoid $\Lambda=\bigcap_{n\geq 0}F^n(\TT)$, and one can show that $\Lambda$ is a hyperbolic attractor for $F$.

We are going to focus on a more particular case of $F$, though we believe that our result can be extended to a more general class of solenoids.

\subsection{A particular case}\label{subsec2.1} Fix an integer $m\geq 2$ and coefficients $\eta\in(0,1)$ and $\alpha,\delta\in(0,\min\{\eta,1-\eta\})$. Consider a diffeomorphism onto its image $F\colon\TT\to\TT$ of the form
$$F(t,x,y)=\left(mt,\alpha x+\eta\cos(2\pi t),\delta y+\eta\sin(2\pi t)\right).$$

It is easy to show that with our choice of coefficients $F(\TT)\subset\INT\TT$, and hence we set $\Lambda=\bigcap_{n=0}^{\infty}F^n(\TT)$, which is a topological attractor for $F$. We also have the following properties:

\begin{itemize}
    \item the set $\Lambda$ is hyperbolic for $F$ and $F|_{\Lambda}$ is topologically transitive;
    \item the set of periodic points of $F$ is dense in $\Lambda$ and hence $NW(F)=\Lambda$;
    \item the map $F$ expands in the $S^1$-direction, and contracts in the $D$-direction, so $E^u(x)$ is one dimensional for every $x\in\Lambda$;
    
    \item the map $F$ has a fixed point $p=\left(0,\frac{\eta}{1-\alpha},0\right)\in\Lambda$.
\end{itemize}

\begin{prop}
Suppose that $\alpha=\delta=:\lambda <1/m$. Then, the map $F$ satisfies the properties (A1)-(A5) as in the setting from Section \ref{Setting}.
\end{prop}

\begin{proof}
We need to verify that in a neighborhood of the point $p$, the map $F$ can be written as the time-one map of a linear vector field. To this end, we need to find a change of coordinates $\phi$ in a neighborhood of $p$ in such a way that the map $f:=\phi\circ F\circ\phi^{-1}$ fixes 0 and in a neighborhood of 0 it is the time-one map of a flow given by
$$\begin{pmatrix}
t\\
x\\
y
\end{pmatrix}'=
\begin{pmatrix}
 \gamma & 0 & 0\\
 0 & -\beta & 0\\
 0 & 0 & -\beta
\end{pmatrix}\begin{pmatrix}
t\\
x\\
y
\end{pmatrix},
$$

as required in Section \ref{Setting}. That is, we need $f(t,x,y)=(e^{\gamma}t,e^{-\beta}x,e^{-\beta}y)$ for some $\beta>\gamma>0$. We set $\gamma:=\log m$ and $\beta:=\log(1/\lambda)$, so the final map must be $f(t,x,y)=(mt,\lambda x,\lambda y)$ in a neighborhood of $0$. Observe that since $\lambda<1/m$, we have $\beta >\gamma >0$.

Let us search for the map $\phi$ in the form $\phi(t,x,y)=(t,x-a(t),y-b(t))$, where $a(t)$ and $b(t)$ are two functions to be chosen. Observe that we have $\phi^{-1}(t,x,y)=(t,x+a(t),y+b(t))$, so 
$$F(\phi^{-1}(t,x,y))=\left(mt,\lambda x+\lambda a(t)+\eta\cos(2\pi t),\lambda y+\lambda b(t)+\eta\sin(2\pi t)\right).$$

Then, applying $\phi$ we get
$$f(t,x,y)=\left(mt,\lambda x+\lambda a(t)+\eta\cos(2\pi t)-a(mt),\lambda y+\lambda b(t)+\eta\sin(2\pi t)-b(mt)\right).$$

In order to satisfy $f(t,x,y)=(mt,\lambda x,\lambda y)$, we need:

\begin{equation}\label{ee1}
     a(mt)=\lambda a(t)+\eta\cos(2\pi t)\qquad\text{and}\qquad b(mt)=\lambda b(t)+\eta\sin(2\pi t).
\end{equation}

 Since we want the functions $a(t),b(t)$ to be at least $C^{\infty}$, we can write them as power series $a(t)=\sum_{k=0}^{\infty}a_kt^k$ and $b(t)=\sum_{k=0}^{\infty}b_kt^k$. Now we need to find the coefficients $a_k, b_k$ so that \eqref{ee1} holds. That is,
 
 \begin{align}
 \sum_{k=0}^{\infty}m^ka_kt^k&=\sum_{k=0}^{\infty}\lambda a_kt^k+\eta\sum_{\{k\geq 0:k\text{ even}\}}\frac{(-1)^{k/2}2^k\pi^k}{k!}t^k,\label{ee2}\\
 \sum_{k=0}^{\infty}m^kb_kt^k&=\sum_{k=0}^{\infty}\lambda b_kt^k+\eta\sum_{\{k\geq 0:k\text{ odd}\}}\frac{(-1)^{\frac{k-1}{2}}2^k\pi^k}{k!}t^k.\label{ee3}
 \end{align}
 
 Comparing coefficients, from \eqref{ee2} we have
 $$a_k=\left\{\begin{array}{lcc}
      0 & , & k\text{ is odd}  \\
      {}\\
      \dfrac{\eta(-1)^{k/2}2^k\pi^k}{(m^k-\lambda)k!}& , & k\text{ is even} 
 \end{array},\right.$$
 
 and from \eqref{ee3} we have
 
 $$b_k=\left\{\begin{array}{lcc}
      \dfrac{\eta(-1)^{\frac{k-1}{2}}2^k\pi^k}{(m^k-\lambda)k!} & , & k\text{ is odd}  \\
      {}\\
      0 & , & k\text{ is even} 
 \end{array}.\right.$$
 
The power series $a(t)$ and $b(t)$ defined with these coefficients have infinite radius of convergence, and in particular they define two $C^{\infty}$ functions, moreover analytic, in a neighborhood of $t=0$. By construction, they satisfy $f(t,x,y)=(mt,\lambda x,\lambda y)$ as desired. The map $\phi$ is clearly a local diffeomorphism at $p$.
\end{proof}

\vspace{0.5cm}

\section{Properties of the map $g$}\label{prelim}

In order to establish the decay of correlations, we need to summarize properties of the slow-down map $g$. Proofs of these results will be omitted and cited.

\subsection{The map $g$ as a Young diffeomorphism}
We now give a brief description of the Young structure provided in \cite[Theorem 4.2]{Z}. Let $f\colon\Lambda\to\Lambda$ and $g\colon\Lambda_g\to\Lambda_g$ be as in Section \ref{Setting}.

There is a homeomorphism $h\colon\Lambda_g\to\Lambda$ with $h\circ g|_{\Lambda_g}=f\circ h$, which satisfies $d_{C^0}(\id|_{\Lambda_g},h)\leq Kr_1$ (recall that $r_1$ is the radius of the slow-down ball) for some constant $K>0$ and $h(p)=p$ (see \cite[Proof of Statement 2 in Theorem 3.1]{Z}). By \cite[Theorem 3.3]{Z}, for every $x\in\Lambda_g\minus\{p\}$, there is a one dimensional subspace $E^u_g(x)$ of $T_xM$, and a $\dim M-1$ dimensional subspace $E^s_g(x)$ of $T_xM$, which depends continuously on $x$ and can be extended by continuity to $x=p$. Moreover, for every $x\in \Lambda_g$ there are $C^{1+\alpha}$ embedded disks $W^u_x$, $W^s_x\subset U$ such that $T_xW^u_x=E^u_g(x)$, $T_xW^s_x=E^s_g(x)$. In fact $W^u_x$ and $W^s_x\cap\Lambda_g$ are images of unstable and stable manifolds of $f$ (intersected with $\Lambda$) under $h^{-1}$.

Let $\mathcal{R}$ be a Markov partition for $f$, then $\mathcal{P}:=h^{-1}(\mathcal{R})$ is a Markov partition for $g$. Given $Q>0$, we can choose $r_1$ small enough so that if $\diam\mathcal{P}$ is small enough, then there is $P\in\PP$ such that: every piece of trajectory of $g$ starting from $P$ and ending in $B(0,r_1)$ has length at least $Q$, and every piece of trajectory of $g$ starting from $B(0,r_1)$ and ending in $P$ has length at least $Q$. To ensure this, take a small $r_1>0$ and $P\in\PP$ which does not intersect $V$ (see condition (A4)). Also, we choose $P$ such that it intersects with $W_0^s$.

The set $P$ has hyperbolic product structure, so it will be the base of the Young structure. One can choose closed segments of local unstable manifolds $W_x^u\cap P$, for $x\in P$, as the continuous family of unstable disks in \ref{young}, and we use the same notation $\Gamma^u=\{V^u\}$. To obtain a continuous family of stable disks, one can choose the connected component of $W^s_x\minus\partial^uP$ containing $x\in P$, where $\partial^uP$ denotes the unstable boundary of $P$ (see for example \cite[Definition 18.7.1]{KH}), and again we use the same notation $\Gamma^s=\{V^s\}$. Consider $\tau(x)$ to be the first return time of $x$ to $P$. For $n\geq 0$, assign a finite collection $\{P^s_{i_k(n)}\}$ of connected components of the set $\{x\in P:\tau(x)=n\}$. This structure satisfies all the conditions for a Young diffeomorphism given in \ref{young} (see \cite[Proof of Theorem 4.2]{Z}). Denote by $G$ the induced map of the map $g$ (see right before Condition (C4) in the definition of Young diffeomorphism).

For the following properties and facts about inducing schemes we refer to \cite{PSZ1}. It was shown that the SRB measure $\mu$ is the lifted measure of a measure $\eta$, which is the equilibrium measure of the normalized induced potential $\overline{\varphi}\colon\bigcup P^s_{i_n}\to\R$ given by
$$\overline{\varphi}(x)=-\sum_{k=0}^{\tau(x)-1}\log\abs{dg|_{E^u_g(g^kx)}}-P_L(\varphi_1)\tau(x),$$

where $P_L$ is the topological pressure defined by the variational principle considering only lifted measures (see \cite{PSZ1} for more details). The measure $\eta$ has a Gibbs property with the following consequence: there is $C_0\geq 1$ such that for every $P^s_{i_n}$ and every $x\in P^s_{i_n}$, we have
\begin{equation}\label{gibbs}
C_0^{-1}\leq\frac{\eta
(P^s_{i_n})}{J^uG(x)^{-1}}\leq C_0.
\end{equation}

Also, by \cite[Lemma 7.2]{PSZ1} there is a constant $C_1>0$ such that for any $P^s_{i_n}$, any $x\in scl(P^s_{i_n})$ and $W\in\Gamma^u$, we have
\begin{equation}\label{postgibbs}
C_1^{-1}J^uG(x)^{-1}\leq\mu_W(W\cap P^s_{i_n})\leq C_1 J^uG(x)^{-1}.
\end{equation}

In order to establish decay of correlations, we are going to use the following proposition, which is a corollary of results in \cite{SZ} (see also \cite{G} and \cite{S}). The proposition is stated as in \cite{PSS}, but adapted to our work.

\begin{prop}\label{bestprop}
    Assume that $\gcd\{\tau\}=1$, and that there is a constant $K_1>0$ such that for any $P^s_{i_n}$, any $x,y\in P^s_{i_n}$ and $0\leq j\leq \tau(P^s_{i_n})$ we have
    \begin{equation}\label{assum1}
    d(g^j(x),g^j(y))\leq K_1\max\{d(x,y),d(G(x),G(y))\}.
    \end{equation}

    Also, assume that there are $\delta>1$ and $K_2>0$ such that 
    \begin{equation}\label{assum2}
    \eta(\tau > n)\leq \frac{K_2}{n^{\delta}}.
    \end{equation}

    Then, the following statements hold:
    \begin{enumerate}
        \item For any $h_1,h_2\in\HH$, there is $K_3>0$ such that $\abs{C^{\mu}_n(h_1,h_2)}\leq K_3n^{-(\delta-1)}$.

        \item There is a sequence $\YY=\{Y_i\}$ of nested subsets of $\Lambda_g$ such that for any $h_1,h_2\in\GG(\YY)$, we have
        $$C^{\mu}_n(h_1,h_2)=\sum_{N=n+1}^{\infty}\eta(\tau>N)\int h_1\dmu\int h_2\dmu+O(R_{\delta}(n)),$$
        where
        $$R_{\delta}(n)=\left\{\begin{array}{lcc}
            1/n^{\delta} & , & \delta >2\\
            \log(n)/n^2 & , & \delta = 2\\
            1/n^{2\delta-2} & , & 1<\delta<2
        \end{array}.\right.$$
    \end{enumerate}
\end{prop}

\section{Proof of Theorems \ref{mainresult1} and \ref{mainresult2}}\label{Proofs}

We aim to use Proposition \ref{bestprop}. To this end we need to bound the $\eta$-measure of the tails of the return time. First, we observe that it suffices to bound the measure of the tail of the return time with respect to the leaf volume $m_W$, where $W$ is an unstable curve in the base $P$ of the Young structure. This is useful as this measure is the length on the leaf. So we fix an unstable leaf $W$ and we denote by $\ell$ the measure $m_W$.

\begin{lem}\label{comparisonetalength}
Let $P^s_i$ be any basic set from the Young structure. Then, there is $C_2\geq 1$ such that 
$$C_2^{-1}\eta(P^s_i)\leq \ell(W\cap P^s_i)\leq C_2\eta(P^s_i).$$
\end{lem}

\begin{proof}
    We start comparing $\ell$ with $\mu_W$ in $W$. These measures are equivalent, and the density $d$ of the conditional measure $\mu_W$ with respect to $\ell$ on $W$ is given by $d(y)=\frac{1}{R}\rho(y)$, where $\rho$ and $R$ are as follows: Fix $z\in W$,
    $$\rho(y):=\lim_{n\to\infty}\prod_{k=0}^{n-1}\frac{\Jac(dg|E^u(g^{-k}y))}{\Jac(dg|E^u(g^{-k}z)},\qquad R:=\int_W\rho\,d\ell.$$

    It is known that the function $\rho$ (and hence the function $d$) depends continuously on $y\in W$. Since $W$ is a closed segment of unstable leaf in a compact set, it is compact as well and hence there is $K_3\geq 1$ such that $K_3^{-1}\leq d(y)\leq K_3$ on $W$. Thus for a measurable set $E\subset W$, integrating the last inequality with respect to $\ell$ over $E$ we obtain 
    $$K_3^{-1}\ell(E)\leq\mu_W(E)\leq K_3\ell(E).$$

    Now consider $P^s_i$. By \eqref{gibbs} and \eqref{postgibbs} we have that
    $$(C_0C_1)^{-1}\eta(P^s_i)\leq\mu_W(W\cap P^s_i)\leq C_0C_1\eta(P^s_i).$$

    The result follows setting $C_2:=C_0C_1K_3$.
\end{proof}

It follows from Lemma \ref{comparisonetalength} that
\begin{equation}\label{largovseta}
C^{-1}_2\eta(\{x:\tau(x)>n\})\leq\ell(\{x\in W:\tau(x)>n\})\leq C_2\eta(\{x:\tau(x)>n\}),
\end{equation}
as the sets of points where $\tau>n$ is the disjoint union of the sets where $\tau =N$ for $N>n$, and due to the choice of the sets $P^s_i$. Because of this fact, it is very useful to study the length of the unstable curves in $P$, and how this length behaves as the curves pass through the slow-down ball.

By \cite[Proof of Statement 1 in Theorem 3.1]{Z}, we have that if $\sigma$ is a piece of unstable curve, and if $r>s$ are so that $g^s(\sigma),g^{s+1}(\sigma),\ldots,g^{r-1}(\sigma)$ are all outside $B(0,r_1)$, then 
\begin{equation}\label{unifexp}
\ell(g^r(\sigma))\geq C_3\nu^{r-s}\ell(g^s(\sigma)),
\end{equation}
where $\nu>1$ is as in
Section \ref{Setting} and $C_3>0$ is a constant (which we may choose to be less than 1).

Also, observe that $\norm{dg}$ is bounded in $\Lambda_g$ as a continuous function in a compact set. Set $\xi:=\sup_{x\in\Lambda_g}\norm{d_xg}<\infty$. Since for $x\in\Lambda_g$ and $v\in E_g^u(x)$ we have $\norm{d_xg(v)}\leq\xi\norm{v}$, it follows that if $\sigma\subset\Lambda_g$ is an unstable curve and if $r>s$, then
\begin{equation}\label{unifupper}
    \ell(g^r(\sigma))\leq \xi^{r-s}\ell(g^s(\sigma)).
\end{equation}

Now, we would like to estimate the length of unstable curves as they pass through the perturbed region $Y:=B(0,r_1)$. Let $x\colon[0,T]\to Y$ be a trajectory of the flow such that enters $Y$ at $t=0$ and exits at time $t=T$. It was shown in \cite{CDP} (and mentioned also in \cite{Z}) that if $x(t)$ and $y(t)$ are trajectories on the same unstable leaf, we have
$$d(x(t),y(t))\leq C''(T+1-t)^{-(1+\frac{1}{\alpha})}d(x(T),y(T)).$$

Similarly, we can write an estimate for the distance between the endpoints of an unstable curve as we apply iterates of the map $g$: for an unstable curve $\sigma$ with endpoints $x,y$ such that $g^n(\sigma)$ enters the slow down ball, and $g^m(\sigma)$ exists it, we have
\begin{equation}\label{ineqdist}
d(g^n(x),g^n(y))\leq C'(m-n)^{-(1+\frac{1}{\alpha})}d(g^m(x),g^m(y)).\end{equation}

We need also a similar reverse inequality. 

\begin{lem}\label{newineq} Let $x\colon[0,T]\to Y$ be a trajectory of the flow such that $x(t)$ enters $Y$ at $t=0$ and exits at time $t=T$. Let $y(t)$ be a trajectory in the same unstable leaf as $x(t)$. Then, there is $\gamma_1>0$ such that
$$d(x(T),y(T))\leq (T+1)^{\gamma_1}d(x(0),y(0)).$$
\end{lem}

\begin{proof}
    Let $v(0)\in E^u_g(x(0))$ and $v(T)=D_{x(0)}\phi_T(v(0))$, where $\phi_t$ is the flow generated by the perturbed system of differential equations (see Section \ref{Setting}). By \cite[Formulas (7.29) and (7.30)]{CDP}, we have
    $$\log\left(\dfrac{\norm{v(T)}}{\norm{v(0)}}\right)=\int_0^T\norm{x}^{\alpha}(\alpha\inner{\hat{v}}{\hat{x}}\inner{\hat{v}}{A\hat{x}}+\inner{\hat{v}}{A\hat{v}})dt,$$

    where $\hat{x}=x/\norm{x}$ and $\hat{v}=v/\norm{v}$. Since $v(t)$ is an unstable vector, its angle $\rho(t)$ with respect to $E^u_g$ is identically zero, and by the computations in \cite[Section 7.3]{CDP} we have
    \begin{equation}\label{inner}
    \inner{\hat{v}}{\hat{x}}\inner{\hat{v}}{A\hat{x}}\leq \gamma\quad\text{and}\quad\inner{\hat{v}}{A\hat{v}}=\gamma,
    \end{equation}
    where $\gamma$ is as in condition (A4) in Section \ref{Setting}. Thus,

    \begin{equation}\label{Q3}
        \log\left(\dfrac{\norm{v(T)}}{\norm{v(0)}}\right)\leq\gamma(\alpha+1)\int_0^T\norm{x}^{\alpha}dt.
    \end{equation}

Let $\theta(t)$ be the positive angle between $x(t)$ and $E^u(0)$ and write $x=x_s+x_u$. Then $\tan\theta=\norm{x_s}/\abs{x_u}$. Let $T_1$ denote the time for which $\tan\theta(T_1)=1$. (It follows from \cite[Lemma 7.1]{CDP} that $\tan\theta$ is strictly decreasing.) 

By \cite[Lemma 5.4]{Z} we have that for $t\leq T_1$, if we denote $\chi:=\alpha\dfrac{\beta-\gamma}{2}$ then
$$\norm{x(t)}^{\alpha}\leq\frac{1}{r_1^{-\alpha}+\chi t}\leq \frac{Q_1}{1+t}.$$

Observe that we can choose $Q_1:=\frac{2}{\alpha(\beta-\gamma)}$. This follows from the fact that we can choose any $Q_1\geq\frac{1+t}{r_1^{-\alpha}+\chi t}$ and the function $F(s)=\frac{1+s}{r_1^{-\alpha}+\chi s}$ is increasing (for $r_1$ small enough) with horizontal asymptote $y=1/\chi$. Hence, we conclude that for $t\leq T_1$,
\begin{equation}\label{Q1}
    \norm{x(t)}^{\alpha}\leq\dfrac{2}{\alpha(\beta-\gamma)}\cdot\dfrac{1}{1+t}.
\end{equation}

Now we have to bound $\norm{x(t)}^{\alpha}$ for $T_1\leq t\leq T$. Since $\tan\theta$ is decreasing, we have that for $t\geq T_1$, $\tan\theta\leq 1$ and therefore $\norm{x_s}\leq\abs{x_u}$. Thus,

$$\norm{x}^{\alpha}=(\norm{x_s}^2+\abs{x_u}^2)^{\alpha/2}\leq 2^{\alpha/2}\abs{x_u}^{\alpha},$$

and therefore it suffices to bound $\abs{x_u}^{\alpha}$. Since $\dot{x}=\norm{x}^{\alpha} Ax$, when we restrict to the $u$-coordinate we have $x_u'=\gamma\norm{x}^{\alpha}x_u$. Since $x_u$ does not vanish, we can consider separately the cases where it is positive, and where it is negative. Then, we obtain $\abs{x_u}'=\gamma\norm{x}^{\alpha}\abs{x_u}$ and it follows that $\abs{x_u}'\geq\gamma\abs{x_u}^{\alpha+1}$. Hence,
$$\left(-\frac{1}{\alpha}\abs{x_u}^{-\alpha}\right)'=\abs{x_u}^{-\alpha-1}\abs{x_u}'\geq\gamma.$$

Integrating over $t\leq\tau\leq T$ we get $-\dfrac{1}{\alpha}(\abs{x_u(T)}^{-\alpha}-\abs{x_u(t)}^{-\alpha})\geq\gamma(T-t)$, and therefore
$$\abs{x_u(t)}^{-\alpha}\geq \abs{x_u(T)}^{-\alpha}+\alpha\gamma(T-t)\geq r_1^{-\alpha}+\alpha\gamma(T-t)$$

as $\abs{x_u(T)}\leq\norm{x(T)}=r_1$. Then, we have that
$$\abs{x_u(t)}^{\alpha}\leq\frac{1}{r_1^{-\alpha}+\alpha\gamma(T-t)}\leq\frac{Q_2}{1+T-t}.$$

Observe that we can choose $Q_2:=\frac{1}{\alpha\gamma}$. This follows from the fact that we can choose any $Q_2\geq\frac{1+T-t}{r_1^{-\alpha}+\alpha\gamma(T-t)}$ and the function $G(s)=\frac{1+s}{r_1^{\alpha}+\alpha\gamma s}$ is increasing (for $r_1$ small enough) with horizontal asymptote $y=1/\alpha\gamma$. Hence, we conclude that for $T_1\leq t\leq T$,
\begin{equation}\label{Q2}
    \norm{x(t)}^{\alpha}\leq \dfrac{2^{\alpha/2}}{\alpha\gamma}\cdot\dfrac{1}{1+T-t}.
\end{equation}

Putting together \eqref{Q3}, \eqref{Q1} and \eqref{Q2} we get
\begin{align*}
    \log\left(\dfrac{\norm{v(T)}}{\norm{v(0)}}\right)&\leq \gamma(\alpha+1)\left(\int_0^{T_1}\norm{x}^{\alpha}dt+\int_{T_1}^T\norm{x}^{\alpha}dt\right)\\
    &\leq\gamma(\alpha+1)\left(\dfrac{2}{\alpha(\beta-\gamma)}\int_0^{T_1}\dfrac{dt}{1+t}+\dfrac{2^{\alpha/2}}{\alpha\gamma}\int_{T_1}^T\dfrac{dt}{1+T-t}\right)\\
    &=\gamma(\alpha+1)\left(\dfrac{2}{\alpha(\beta-\gamma)}\log(1+T_1)+\dfrac{2^{\alpha/2}}{\alpha\gamma}\log(1+T-T_1)\right)\\
    &\leq\gamma(\alpha+1)\left(\dfrac{2}{\alpha(\beta-\gamma)}+\dfrac{2^{\alpha/2}}{\alpha\gamma}\right)\log(T+1).
\end{align*}

Thus, for a trajectory $y\in W^u_x$ we have that $d(x(T),y(T))\leq (T+1)^{\gamma_1}d(x(0),y(0))$, where $\gamma_1=\gamma(\alpha+1)\left(\frac{2}{\alpha(\beta-\gamma)}+\frac{2^{\alpha/2}}{\alpha\gamma}\right)$.\qedhere
\end{proof}

Similarly as in \eqref{ineqdist}, it follows from Lemma \ref{newineq} that for an unstable curve $\sigma$ with endpoints $x,y$ such that $g^n(\sigma)$ enters the slow down ball, and $g^m(\sigma)$ exists it, we have
\begin{equation}\label{ineqdist2}
d(g^m(x),g^m(y))\leq C_4(m-n)^{\gamma_1}d(g^n(x),g^n(y)).
\end{equation}

\begin{lem}\label{lemalargos}
Let $\sigma\subset W$ be a connected component of $\{\tau=N\}\cap W$. Let $x,y$ be its endpoints. Assume that the curve enters the slow down ball at time $n$ under iterates of $g$, and exits it at time $m$, where $n<m<N$. Then, there are constants $C_5,C_6>0$ and $\gamma_1,\gamma_2>0$ which do not depend on $n,m,N$ such that
\begin{equation}\label{expslowdown}
C_5(m-n)^{\gamma_2}\ell(g^n(\sigma))\leq\ell(g^m(\sigma))\leq C_6(m-n)^{\gamma_1}\ell(g^n(\sigma)).
\end{equation}
\end{lem}
\begin{proof}
We follow the idea of the proof in \cite[Lemma 6.9]{PSS}. Since we are choosing $\diam\PP$ small, by uniform continuity of $g^{-1}|_{\Lambda_g}$, there is a constant $K\geq 1$ such that for every $V^u\in\Gamma^u$ with endpoints $w_1,w_2$ we have
\begin{equation}\label{ineqlargodistance}
K^{-1}d(g^{-k}(w_1),g^{-k}(w_2))\leq\ell(g^{-k}(V^u))\leq Kd(g^{-k}(w_1),g^{-k}(w_2))
\end{equation}
for every $k\geq 0$. In particular, this holds for the curves $g^{N-k}(\sigma)$ and its endpoints $g^{N-k}(x)$, $g^{N-k}(y)$.

For the lower bound, set $\gamma_2=1+\frac{1}{\alpha}$. Since $n<m<N$, combining \eqref{ineqdist} and \eqref{ineqlargodistance} we have
\begin{align*}
    \ell(g^m(\sigma))&\geq K^{-1}d(g^m(x),g^m(y))\\
    &\geq K^{-1}(C')^{-1}(m-n)^{\gamma_2}d(g^n(x),g^n(y))\\
    &\geq K^{-2}(C')^{-1}(m-n)^{\gamma_2}\ell(g^n(\sigma)).
\end{align*}
Observe that we can choose $C_5:=K^{-2}(C')^{-1}\leq 1$.

For the upper bound, set $\gamma_1$ as in Lemma \ref{newineq}. Now combining \eqref{ineqdist2} and \eqref{ineqlargodistance} we have
\begin{align*}
    \ell(g^m(\sigma))&\leq K d(g^m(x),g^m(y))\\
    &\leq KC_4(m-n)^{\gamma_1}d(g^n(x),g^n(y))\\
    &\leq K^2C_4(m-n)^{\gamma_1}\ell(g^n(\sigma)).
\end{align*}
Finally, we can choose $C_6:=K^2C_4\geq 1$.
\end{proof}

\subsection{An upper bound for the tail of the return time}

Let $\sigma\subset W$ be a connected component of $\{x\in W:\tau(x)=N\}$. As a consequence of the Markov property (C1b) in the definition of Young diffeomorphism, $g^N(\sigma)$ is a full unstable curve on $P$, and so $\ell(g^N(\sigma))\leq\sup_{V^u}\ell(V^u)=:M$.

Following the line of arguments in \cite{PSS}, we study the trajectory of $\sigma$ depending on how many times it enters the slow down ball, and how long these pieces of trajectory lie inside and outside the ball. 

Choose numbers $k=k(\sigma)$, $p=p(\sigma)$ and finite collections of nonnegative numbers $\{k_i\}_{i=1}^p$ and $\{l_i\}_{i=0}^p$ such that:

\begin{itemize}
    \item $\sum_{i=1}^pk_i=k$ and $\sum_{i=0}^pl_i=N-k$; and
    \item the trajectory of the set $\sigma$ under $g$ consecutively spends $l_i$-times outside $B(0,r_1)$ and $k_i$-times inside $B(0,r_1)$.
\end{itemize}

In order to have lighter notation in the inequalities, we recall that $Y=B(0,r_1)$, $\gamma_2=1+\frac{1}{\alpha}$, and define $D_r=\sum_{i=1}^rk_i$, $E_r=\sum_{i=0}^rl_i$ and $L(r)=\ell(g^r(\sigma))$. The following lemma is a simple consequence of \eqref{unifexp} and \eqref{expslowdown}.

\begin{lem}\label{lemainduccion}
For each $1\leq j\leq p$, we have
\begin{equation}\label{UpperBound}
    L(N)\geq C_3^j\nu^{\sum_{i=0}^{j-1}l_{p-i}-1}C_5^{j}\left(\prod_{i=0}^{j-1}k_{p-i}\right)^{\gamma_2}L(E_{p-j}+D_{p-j}+1).
\end{equation}
\end{lem}

\begin{proof}
We prove Lemma \ref{lemainduccion} by an induction argument (up to $p$). We first show that $\eqref{UpperBound}$ holds for $j=1$. Observe that by the choice of the $l_i$'s, the iterates of $\sigma$ under $g$ spend $l_p$ times outside $Y$ right before returning to $P$ (at time $N$), so by \eqref{unifexp} we have
\begin{equation}\label{casobase1}
    L(N)\geq C_3\nu^{l_p-1}L(D_p+E_{p-1}+1).
\end{equation}

Now since $\sigma$ exits $Y$ at time $D_p+E_{p-1}+1$ and before that it enters at time $E_{p-1}+D_{p-1}+1$, by \eqref{expslowdown}, we have
\begin{equation}\label{casobase2}
    L(E_{p-1}+D_p+1)\geq C_5k_p^{\gamma_2}L(E_{p-1}+D_{p-1}+1).
\end{equation}

Putting together \eqref{casobase1} and \eqref{casobase2}, we obtain
\begin{equation*}
    L(N)\geq C_3\nu^{l_p-1}C_5k_p^{\gamma_2}L(E_{p-1}+D_{p-1}+1),
\end{equation*}
as desired for the case $j=1$.

For the inductive argument, assume \eqref{UpperBound} holds for some $1\leq j\leq p-1$. We show that it also holds for $j+1$. Observe that $E_{p-j}+D_{p-j}+1$ is a time when $\sigma$ enters $Y$. Then, the iterates of $\sigma$ between the times $E_{p-j}+D_{p-j}$ and $E_{p-j-1}+D_{p-j}+1$ are all outside $Y$. By \eqref{unifexp}, we get
\begin{equation}\label{casoinductivo1}
    L(E_{p-j}+D_{p-j}+1)\geq C_3\nu^{l_{p-j}}L(E_{p-j-1}+D_{p-j}+1).
\end{equation}

Now, since $\sigma$ exits $Y$ at time $E_{p-j-1}+D_{p-j}+1$ and before that it entered at time $E_{p-j-1}+D_{p-j-1}+1$, by \eqref{expslowdown} we have
\begin{equation}\label{casoinductivo2}
    L(E_{p-j-1}+D_{p-j}+1)\geq C_5k_{p-j}^{\gamma_2}L(E_{p-(j+1)}+D_{p-(j+1)}+1).
\end{equation}

Putting together the induction hypothesis for $j$ \eqref{UpperBound}, \eqref{casoinductivo1} and \eqref{casoinductivo2} we obtain
\begin{equation*}
    L(N)\geq C_3^{j+1}\nu^{\sum_{i=0}^jl_{p-i}-1}C_5^{j+1}\left(\prod_{i=0}^jk_{p-i}\right)^{\gamma_2}L(E_{p-(j+1)}+D_{p-(j+1)}+1),
\end{equation*}

which concludes the inductive step and the proof of Lemma \ref{lemainduccion}.
\end{proof}

Observe that for $j=p$ in Lemma \ref{lemainduccion} we have the estimate
$$L(N)\geq C_3^p\nu^{\sum_{i=1}^pl_i-1}C_5^{p}\left(\prod_{i=1}^pk_i\right)^{\gamma_2}L(l_0+1).$$

Using \eqref{unifexp} one more time, we get
$\ell(g^N(\sigma))\geq C_3^{p+1}\nu^{N-k}C_5^{p}\left(\prod_{i=1}^pk_i\right)^{\gamma_2}\ell(\sigma)$.

Since $\prod k_i\geq k$ (as $k_i\geq 2$, make $r_1$ smaller if necessary), we have that
$$\ell(\sigma)\leq C_3^{-(p+1)}\nu^{-(N-k)}(C_5)^{-p}k^{-\gamma_2}M.$$

Recall that we can choose the constants $C_3\leq1$ and $C_5\leq1$. So $B:=(C_3C_5)^{-1}\geq1$. By setting $C_7:=C_3^{-1}M$ we have that 
$$\ell(\sigma)\leq C_7B^p\nu^{-(N-k)}k^{-\gamma_2}.$$

Observe that by the choice of $Q$ we have 
\begin{equation}\label{choiceofQ}
N=k+\sum_{i=0}^{p}l_i\geq k+(p+1)Q,
\end{equation}
and so $p\leq\dfrac{N-k}{Q}$. Thus, 
$$B^p=e^{p\log B}\leq e^{\frac{N-k}{Q}\log B}<e^{\varepsilon_0(N-k)}$$
for sufficiently small $\varepsilon_0>0$ and choosing $Q$ large. Then,

$$\ell(\sigma)\leq C_7e^{(\varepsilon_0-\log\nu)(N-k)}k^{-\gamma_2}.$$

\begin{lem}\label{lemcard}
Define $\mathcal{S}_{k,N,p}:=\{\sigma\subset W:\tau(\sigma)=N,k(\sigma)=k,p(\sigma)=p\}$. Then, there are $0<h<h_{\text{top}}(g)$, $0<\varepsilon_0<h_{\text{top}}(g)-h$ and $C_8>0$ such that
$$\#\mathcal{S}_{k,N,p}\leq C_8\frac{1}{p^2}e^{(h+\varepsilon_0)(N-k)}.$$
\end{lem}

\begin{proof}
The proof of Lemma \ref{lemcard} is similar to \cite[Lemma 8.1]{PSS}. We consider the shift space that models the system $(\Lambda_g,g)$. Call $P^*$ the element of the Markov partition containing the ball $B(0,r_1)$. Observe that $\#\S_{k,N,p}$ is bounded above by the number of words of length $N$ that start and end at $P$ and contain exactly $k$ symbols $P^*$. With the previous notation we see one of these words in Figure 1.

\begin{figure}[h]
\centering
\tikzset{every picture/.style={line width=0.75pt}} 

\begin{tikzpicture}[x=0.5pt,y=0.5pt,yscale=-1,xscale=1]

\draw   (7,129.33) -- (657,129.33) -- (657,157.33) -- (7,157.33) -- cycle ;
\draw    (24,129.44) -- (24.11,157.44) ;
\draw    (94.61,129.67) -- (94.72,157.67) ;
\draw    (139.39,129.44) -- (139.5,157.44) ;
\draw    (222.33,129.67) -- (222.44,157.67) ;
\draw    (269.83,129.17) -- (269.94,157.17) ;
\draw    (638.28,129.61) -- (638.39,157.61) ;
\draw    (578,129.67) -- (578.11,157.67) ;
\draw    (522.5,129.67) -- (522.61,157.67) ;
\draw   (13,162) .. controls (13.05,166.67) and (15.4,168.98) .. (20.07,168.93) -- (41.07,168.72) .. controls (47.74,168.65) and (51.09,170.95) .. (51.14,175.62) .. controls (51.09,170.95) and (54.4,168.59) .. (61.07,168.52)(58.07,168.55) -- (82.07,168.31) .. controls (86.74,168.26) and (89.05,165.91) .. (89,161.24) ;
\draw   (100.56,161.44) .. controls (100.5,165.81) and (102.65,168.03) .. (107.02,168.08) -- (107.02,168.08) .. controls (113.27,168.17) and (116.36,170.39) .. (116.3,174.76) .. controls (116.36,170.39) and (119.51,168.25) .. (125.75,168.33)(122.94,168.29) -- (125.75,168.33) .. controls (130.12,168.39) and (132.33,166.23) .. (132.39,161.86) ;
\draw   (142.78,161.11) .. controls (142.83,165.78) and (145.18,168.09) .. (149.85,168.04) -- (170.85,167.83) .. controls (177.52,167.76) and (180.87,170.06) .. (180.92,174.73) .. controls (180.87,170.06) and (184.18,167.7) .. (190.85,167.63)(187.85,167.66) -- (211.85,167.42) .. controls (216.52,167.37) and (218.83,165.02) .. (218.78,160.36) ;
\draw   (227.67,160.56) .. controls (227.67,165.23) and (230,167.56) .. (234.67,167.56) -- (236.44,167.56) .. controls (243.11,167.56) and (246.44,169.89) .. (246.44,174.56) .. controls (246.44,169.89) and (249.77,167.56) .. (256.44,167.56)(253.44,167.56) -- (258.22,167.56) .. controls (262.89,167.56) and (265.22,165.23) .. (265.22,160.56) ;
\draw   (531.17,160.78) .. controls (531.17,165.45) and (533.5,167.78) .. (538.17,167.78) -- (539.94,167.78) .. controls (546.61,167.78) and (549.94,170.11) .. (549.94,174.78) .. controls (549.94,170.11) and (553.27,167.78) .. (559.94,167.78)(556.94,167.78) -- (561.72,167.78) .. controls (566.39,167.78) and (568.72,165.45) .. (568.72,160.78) ;
\draw   (582.22,161.1) .. controls (582.27,165.77) and (584.63,168.07) .. (589.3,168.02) -- (605.14,167.84) .. controls (611.81,167.76) and (615.17,170.05) .. (615.22,174.72) .. controls (615.17,170.05) and (618.47,167.68) .. (625.14,167.61)(622.14,167.64) -- (640.97,167.42) .. controls (645.64,167.37) and (647.94,165.01) .. (647.89,160.34) ;

\draw (6,134.72) node [anchor=north west][inner sep=0.75pt] [font=\small]   {$P$};
\draw (638.72,134.33) node [anchor=north west][inner sep=0.75pt]  [font=\small]  {$P$};
\draw (44.5,178.5) node [anchor=north west][inner sep=0.75pt]  [font=\scriptsize]  {$l_{0}$};
\draw (111.83,178.44) node [anchor=north west][inner sep=0.75pt]  [font=\scriptsize]  {$k_{1}$};
\draw (99.9,133.44) node [anchor=north west][inner sep=0.75pt]   {$P^{*}$};
\draw (175.39,178.22) node [anchor=north west][inner sep=0.75pt]  [font=\scriptsize]  {$l_{1}$};
\draw (231.22,133.44) node [anchor=north west][inner sep=0.75pt]  {$P^{*}$};
\draw (240.06,177.56) node [anchor=north west][inner sep=0.75pt]  [font=\scriptsize]  {$k_{2}$};
\draw (543.56,177.78) node [anchor=north west][inner sep=0.75pt]  [font=\scriptsize]  {$k_{p}$};
\draw (609.58,178.56) node [anchor=north west][inner sep=0.75pt]  [font=\scriptsize]  {$l_{p}$};
\draw (533.22,133.44) node [anchor=north west][inner sep=0.75pt]  {$P^{*}$};
\draw (368,135.44) node [anchor=north west][inner sep=0.75pt]  [font=\LARGE]  {$\cdots $};
\end{tikzpicture}
\caption{Trajectory of some $\sigma\in\S_{k,N,p}$.}
\end{figure}

Denote by $\LL$ the language on the shift space, and by $\LL_n$ the set of words of length $n$. Fixing the positions of the $l_i$'s and the $k_i$'s, the number of such words is bounded above by $\#\LL_{l_0}\#\LL_{l_1}\ldots\#\LL_{l_p}$. Since the shift space is topologically mixing, it satisfies the specification property, i.e. there is $r\in\N$ such that for every $v,w\in\LL$, there is $u\in\LL$ with $\abs{u}\leq r$ such that $vuw\in\LL$. It was proven in \cite[Lemma 1.2.2.1]{CT} that $\#\LL_m\#\LL_n\leq(r+1)\#\LL_{m+n+r}$. Using this property inductively, we have that
$$\prod_{j=0}^p\#\LL_{l_j}\leq(r+1)^p\#\LL_{N-k+r}.$$

By \cite[Corollary 1.9.12 and Proposition 3.2.5]{KH}, the number $\#\LL_{n-k+r}$ grows exponentially with an exponent less than or equal to $(N-k)h$ for some $0<h<h_{top}(g)$. Since $g$ and $f_0$ are topologically conjugate, the numbers $r$ and $h$ depend only on $f_0$. 
 
The number of different ways the iterates of some $\sigma\in\S_{k,N,p}$ can enter $Y$ exactly $p$ times and stay inside the ball exactly $k$ times is at most the number of ways one can write $k$ as the sum of $p$ natural numbers (the order matters), which equals $\binom{k-1}{p-1}$. On the other hand, the number of different ways the iterates of some $\sigma\in\S_{k,N,p}$ can spend outside $Y$ exactly $p+1$ times equals the number of of ways the number $N-k$ can be written as the sum of $p+1$ positive natural numbers, which is $\binom{N-k-1}{p}$. Then, there is a constant $C_8>0$ such that
$$\#\S_{k,N,p}\leq C_8\binom{k-1}{p-1}\binom{N-k-1}{p}(r+1)^pe^{h(N-k)}.$$

We claim that there is $0<\eps_0<h_{top}(g)-h$ such that $\binom{k-1}{p-1}<e^{\eps_0(N-k)}$. To this end, observe that by the choice of $Q$ (see Section \ref{prelim} and \eqref{choiceofQ}) we have that $p+1<(N-k)/Q$. Using this property, it can be shown (see \cite[Lemma 8.1]{PSS}) that
$$\binom{k-1}{p-1}\leq e^{\frac{N-k}{Q}\log(4)}\qquad\text{and}\qquad\binom{N-k-1}{p}\leq e^{\frac{N-k}{Q}\log(Qe)}.$$

Finally, observe that
$$p^2(r+1)^p=e^{2\log p+p\log(r+1)}\leq e^{2p+p\log(r+1)}<e^{\frac{(N-k)}{Q}(2+\log(r+1))}.$$

Now given any sufficiently small $\eps_0>0$, we can choose $Q>0$ large enough so that
$\frac{\log(4)+\log(Qe)+2+\log(r+1)}{Q}<\eps_0$. Thus, we conclude that
\begin{align*}
    \#\S_{k,N,p}&\leq C_8\binom{k-1}{p-1}\binom{N-k-1}{p}e^{h(N-k)}\\
    &\leq C_8\frac{1}{p^2}e^{(N-k)\frac{\log(4)+\log(Qe)+2+\log(r+1)}{Q}}e^{h(N-k)}\\
    &\leq C_8\frac{1}{p^2}e^{(h+\eps_0)(N-k)}.\qedhere
\end{align*}
\end{proof}

\begin{lem}\label{lemlognu}
For $f$ in a sufficiently small $C^1$-neighborhood of $f_0$, the expansion rate $\nu$ of $f$ satisfies $\log\nu > h$.
\end{lem}

\begin{proof}
For $f$ in a small neighborhood of $f_0$, these maps are topologically conjugate, and hence they have the same topological entropy. Denote by $\mu_0$ the SRB measure for $f_0$, which is also its measure of maximal entropy by assumption. Since $\mu_0$ is ergodic and $f_0$ has uniform expansion, by Pesin entropy formula we have that
$$\log\lambda=\int\chi_{f_0}^+(x)d\mu_0(x)=h_{\mu_0}(f_0)=h_{top}(f_0)=h_{top}(f)=h_{top}(g)>h.$$

Now, recall that $h$ from Lemma \ref{lemcard} depends only on the symbolic representation of $g$, which depends only on $f_0$ as they are topologically conjugate. Hence, we can choose $f$ sufficiently $C^1$-close to $f_0$ so that $\log\nu>h$, as desired.
\end{proof}

\begin{lem}\label{finalLema}
There is $C_9>0$ such that 
$$\ell(\{x\in W:\tau(x)>n\})\leq C_9n^{-(\gamma_2-1)}.$$
\end{lem}

\begin{proof}
Observe that
$$\ell(\{x\in W:\tau(x)=N\})\leq\sum_{k=1}^N\sum_{p=1}^k\max_{A\in\mathcal{S}_{k,N,p}}\{\ell(A)\}\#\mathcal{S}_{k,N,p}.$$

Then using Lemma \ref{lemcard} we have
\begin{align*}
    \ell(\tau=N)&\leq\sum_{k=1}^N\sum_{p=1}^kC_7e^{(\varepsilon_0-\log\nu)(N-k)}k^{-\gamma_2}C_8\frac{1}{p^2}e^{(h+\varepsilon_0)(N-k)}\\
    &\leq C_7C_8\frac{\pi^2}{6}e^{N\chi}\sum_{k=1}^Ne^{-\chi k}k^{-\gamma_2},
\end{align*}
where $\chi = 2\varepsilon_0-\log\nu+h<0$ (by Lemma \ref{lemlognu}, we can choose $\eps_0$ small enough so that $\chi <0$).

We now use the following criterion for proving the convergence of a sequence.

\begin{prop}\label{stolz}
(Stolz-Cesaro Theorem) Let $(a_n)$, $(b_n)$ be two sequences of real numbers such that $(b_n)$ is strictly monotone and divergent, and the limit
$$\lim_{n\to\infty}\frac{a_{n+1}-a_n}{b_{n+1}-b_n}=L$$
exists. Then, 
$$\lim_{n\to\infty}\frac{a_n}{b_n}=L.$$
\end{prop}

Set $u_k=e^{-\chi k}k^{-\gamma_2}$. Let $a_N=\sum_{k=1}^N(u_{k+1}-u_k)$ and $b_N=\sum_{k=1}^Nu_k$. Observe that $b_N$ is increasing and diverges, and also we have
$$a_{N+1}-a_N=u_{N+2}-u_{N+1}\sim u_{N+1}=b_{N+1}-b_N.$$

Then, we have that 
$$\sum_{k=1}^Nu_k=b_N\sim a_N=\sum_{k=1}^N(u_{k+1}-u_k)=u_{N+1}-u_1\sim e^{-\chi N}N^{-\gamma_2}.$$

Therefore, there is $C_{10}>0$ such that
$$\ell(\tau=N)\leq C_{10}N^{-\gamma_2},$$

and hence there is $C_9>0$ such that
\begin{equation*}
\ell(\{x\in W:\tau(x)>n\})=\sum_{N=n+1}^{\infty}\ell(\tau=N)\leq C_9 n^{-\gamma_2 +1}.\qedhere
\end{equation*}
\end{proof}

\subsection{A lower bound for the tail of the return time} In order to prove the lower bound for the tail of the return time, we need the following Lemma (see \cite[Lemma 7.1]{PSS}). Let $\mathcal{N}=\{n\in\N:\text{there is }x\in P\text{ with }\tau(x)=n\}$.

\begin{lem}\label{Lema7.1PSS} There is an integer $Q_1>1$ such that for any $N>0$ there is $n>N$ with $n\in\mathcal{N}$, an $s$-subset $P^s_l$ with $\tau(P^s_l)=n$ and numbers $0<m_1<m_2$ satisfying $m_1<Q_1$, $n-m_2<Q_1$ such that
\begin{itemize}
    \item $g^k(P^s_l)\cap Y=\vacio$ for $0\leq k<m_1$ or $m_2<k\leq n$; and
    \item $g^k(P^s_l)\cap Y\neq\vacio$ for $m_1\leq k\leq m_2$.
\end{itemize}    
\end{lem}

\begin{proof}
Observe that it is enough to show that there is $Q_1>0$ such that for any $N>0$ we can find an admissible word of length $n>N$ with $n\in\mathcal{N}$ of the form
$$PW_1\claus{P^{*}}W_2P,$$
where the words $W_i$ have length $\abs{W_i}<Q_1$ and do not contain the symbol $P^*$, and the word $\claus{P^*}$ consists of the symbol $P^*$ repeated $n-2-\abs{W_1}-\abs{W_2}$ times.

    Now we need to make a choice on the element of the Markov partition we use as base of the Young structure. Let $z\in W^s_0\minus V$. Choose the element $P$ of the Markov partition for $g$ such that contains $z$ and does not intersect $V$ (see condition (A4) and \cite[Proof of Theorem 4.2]{Z}). Now set $B:=W^s_0\cap P$ and for $x\in B$, define
    $$n(x):=\min\{k\geq 1:g^k(V^s(x))\subset Y\}.$$

    Let $\mathcal{M}:=\{n(x):x\in B\}$ and choose $n_0:=\min\mathcal{M}$. Then there is $x_0\in B$ with $n(x_0)=n_0$. Choose $\gamma^s=V^s(x_0)$ the ``first intersection" (see \cite[Lemma 7.1]{PSS}) between $W^s_0$ and $P$. Similarly, let $\gamma^u$ be the first intersection between $W^u_0$ and $P$. The argument now follows the same line as in \cite{PSS}. The sets $\gamma^s$ and $\gamma^u$ enter $Y$ after finitely many iterations of the maps $g$ and $g^{-1}$ respectively. This way we are able to produce the words
    $$PW_1P^*\quad\text{and}\quad P^*W_2P,$$
    where the lengths $\abs{W_i}$ are bounded above by some $Q_1>1$. Since the symbol $P^*$ can follow itself, we are able to produce a word of the form 
    $$PW_1\claus{P^*}W_2P$$
    of arbitrary large length $n$. Since the word starts and ends at $P$, such $n$ must belong to $\mathcal{N}$. It is important to notice that since $B$ never returns to $P$, there will be sets $P^s_i$ accumulating near $B$ for arbitrarily large $n$. This completes the proof of the lemma.
\end{proof}

\begin{lem}\label{polynlower}
    There is $C_{11}>0$ such that
    $$\ell(\{x\in W:\tau(x)>n\})\geq C_{11}n^{-(\gamma_1-1)}.$$
\end{lem}

\begin{proof}

Denote by $A_l:=P^s_l\cap W$, where $P^s_l$ is given by the previous lemma. Observe that
\begin{align*}
    \ell(\{x\in W:\tau(x)>n\})&=\sum_{N=n+1}^{\infty}\ell(\{x\in W:\tau(x)=N\})\\
    &=\sum_{N=n+1}^{\infty}\sum_{\{i:\tau(P^s_i)=N\}}\ell(P^s_i\cap W)\\
    &\geq\sum_{N=n+1}^{\infty}\ell(A_l).
\end{align*}

Set $I:=\inf_{V\in\Gamma^u}\ell(V)>0$. By Lemma \ref{Lema7.1PSS} (notice that now $\tau(A_l)=N$), using \eqref{unifupper} and \eqref{expslowdown} we have that
\begin{align*}
    I &\leq \ell(g^N(A_l))\leq \xi^{N-m_2-1}\ell(g^{m_2+1}(A_l))\\
    &\leq\xi^{Q_1}C_6(m_2+1-m_1)^{\gamma_1}\ell(g^{m_1}(A_l))
    \leq\xi^{2Q_1}C_6N^{\gamma_1}\ell(A_l).
\end{align*}
Then, there is $C_{11}>0$ such that
\begin{equation*}
\ell(\{x\in W:\tau(x)>n\})\geq I\xi^{-2Q_1}C_6^{-1}\sum_{N=n+1}^{\infty}N^{-\gamma_1}\geq C_{11}n^{-(\gamma_1-1)}.\qedhere
\end{equation*}
\end{proof}

Now we proceed to prove the main results Theorem \ref{mainresult1} and Theorem \ref{mainresult2}. In order to apply Proposition \ref{bestprop}, we need to verify the assumptions. The facts that $\gcd\{\tau\}=1$ and the inequality \eqref{assum1} were shown in \cite[Proofs of Theorems A and B]{Z}.

\begin{proof}[Proof of Theorem \ref{mainresult1}]
The inequality \eqref{assum2} follows from \eqref{largovseta} and Lemma \ref{finalLema} with $K_2=C_2C_9$ and $\delta=\gamma_2-1=1/\alpha>1$. Theorem \ref{mainresult1} follows from Proposition \ref{bestprop} with $s_1=\delta-1=\frac{1}{\alpha}-1>0$.
\end{proof}

\begin{proof}[Proof of Theorem \ref{mainresult2}] Since $\alpha\in(0,1/2)$ (hence $\delta>2$), by Proposition \ref{bestprop} we have that for every $h_1,h_2\in\GG(\YY)$ with $\int h_1\dmu\int h_2\dmu>0$,
$$C_n^{\mu}(h_1,h_2)=\sum_{N=n+1}^{\infty}\eta(\tau>N)\int h_1\dmu\int h_2\dmu+O(1/n^{\gamma_2-1}).$$

It follows from \eqref{largovseta} and Lemma \ref{polynlower} that 
$$\abs{C^{\mu}_n(h_1,h_2)}\geq D_1\sum_{N=n+1}^{\infty}\dfrac{1}{n^{\gamma_1-1}}-\frac{D_2}{n^{\gamma_2-1}}\geq\frac{D_3}{n^{\gamma_1-2}}-\frac{D_2}{n^{\gamma_2-1}}.$$

Observe that we need $0<\gamma_1-2<\gamma_2-1$. The first inequality follows from
$$\gamma_1=\gamma(\alpha+1)\left(\frac{2}{\alpha(\beta-\gamma)}+\frac{2^{\alpha/2}}{\alpha\gamma}\right)=\underbrace{\left(1+\frac{1}{\alpha}\right)}_{\geq 3}\underbrace{\left(\frac{2\gamma}{\beta-\gamma}+2^{\alpha/2}\right)}_{\geq 1}.$$

The second inequality is equivalent to $\gamma_1<\gamma_2+1$, which is equivalent to 
$$(\alpha+1)\left(\frac{2\gamma}{\beta-\gamma}+2^{\alpha/2}\right)<2\alpha+1.$$

If we fix $\gamma>0$ and let $\beta\to\infty$, the left hand side converges to $(\alpha+1)2^{\alpha/2}$, which is strictly less than $2\alpha+1$ for $\alpha\in(0,1/2)$. Thus, there is $\beta_0>\gamma$ such that for every $\beta>\beta_0$ we have $\gamma_1-2<\gamma_2-1$. We conclude that if $\beta>\beta_0$, there is some $D_4>0$ such that
$$\abs{C_{n}^{\mu}(h_1,h_2)}\geq\frac{D_4}{n^{\gamma_1-2}},$$

and the desired result follows with $s_2=\gamma_1-2>0$.\qedhere
\end{proof}

\end{document}